\documentclass[12pt, reqno]{amsart}
\usepackage{amsmath, amsthm, amscd, amsfonts, amssymb, graphicx, color,tikz}
\usepackage[bookmarksnumbered, colorlinks, plainpages]{hyperref}
\usepackage{mathrsfs}
\usepackage{enumerate}

\newtheorem{theorem}{Theorem}[section]

\newtheorem{proposition}[theorem]{Proposition}
\newtheorem{corollary}[theorem]{Corollary}
\theoremstyle{definition}

\newtheorem{remark}[theorem]{Remark}
\newtheorem{question}[theorem]{Question}
\numberwithin{equation}{section}

\newcommand{\A}{\mathrm{Re\,}}

\newcommand{\veps}{\varepsilon}
\newcommand{\CC}{\mathbb C}
\newcommand{\RR}{\mathbb R}
\newcommand{\DD}{\mathbb D}
\newcommand{\QQ}{\mathbb Q}
\newcommand{\TT}{\mathbb T}

\newcommand{\Gu}{{\mathcal G}_{\geq 1}}
\newcommand{\Nphi}{\mathcal N_{\varphi}}

\allowdisplaybreaks[4]

\begin{document}
\setcounter{page}{1}

\title[Counting functions]
{Counting functions for Dirichlet series and compactness of composition operators}
\date{\today}

\author[F. Bayart]{Frédéric Bayart}

\address{Laboratoire de Math\'ematiques Blaise Pascal UMR 6620 CNRS, Universit\'e Clermont Auvergne, Campus universitaire des C\'ezeaux, 3 place Vasarely, 63178 Aubi\`ere Cedex, France.}
\email{frederic.bayart@uca.fr}




\begin{abstract}
We give a sufficient condition for a composition operator with positive characteristic to be compact on the Hardy space of Dirichlet series. 
\end{abstract}
\maketitle

\section{Introduction}

\subsection{Description of the results}

Let $\mathcal H^2$ be the Hilbert space of Dirichlet series $f(s)=\sum_{n\geq 1}a_n n^{-s}$ with square summable coefficients
endowed with 
$\|f\|^2=\sum_{n\geq 1}|a_n|^2.$
By the Cauchy-Schwarz inequality, Dirichlet series in $\mathcal H^2$ generate holomorphic functions in $\CC_{1/2}$,
where $\CC_\theta=\{s\in\CC:\ \Re e(s)>\theta\}$. Let $\varphi:\CC_{1/2}\to\CC_{1/2}$ be analytic. The composition operator with symbol $\varphi$ is defined on $\mathcal H^2$
by $C_\varphi(f)=f\circ\varphi$. In \cite{GorHe}, Gordon and Hedenmalm determined which symbols $\varphi$ generate a bounded composition operator on $\mathcal H^2$:
this happens if and only if $\varphi$ belongs to the Gordon-Hedenmalm class $\mathcal G$ of the analytic functions $\varphi:\CC_{1/2}\to\CC_{1/2}$ which may be written
$\varphi(s)=c_0+\psi(s)$ where $c_0$ is a non-negative integer, $\psi$ is a Dirichlet series that converges uniformly in $\CC_{\veps}$ for every $\veps>0$ and satisfies
the following properties:
\begin{enumerate}[(a)]
 \item if $c_0=0$, then $\psi_0(\CC_0)\subset\CC_{1/2}$;
 \item if $c_0\geq 1$, then either $\psi(\CC_0)\subset\CC_0$ or $\psi_0\equiv 0$.
\end{enumerate}
The non-negative integer $c_0$ is called the characteristic of $\varphi$ and we will use the notation $\mathcal G_0$ and $\mathcal G_{\geq 1}$, respectively,
for the subclasses (a) and (b).

Once you know your operator is continuous the next step is to study whether it is compact. In our context we try to characterize compactness of $C_\varphi$
from properties of its symbol $\varphi$. This has been investigated in many papers (like \cite{Bail15}, \cite{BAYILLI}, \cite{BB16},  \cite{BPortho}, \cite{BP21},  \cite{FQV}, \cite{QS14}).
Following the seminal paper of Shapiro \cite{Sha87} for composition operators on $H^2(\DD)$, a natural way for doing so is to characterize compactness of $C_\varphi$ by mean of some counting function
related to $\varphi$. This was recently achieved in \cite{BP21} for the subclass $\mathcal G_{0}$ of composition operators with zero-characteristic.

In the present paper, we mostly concentrate on the subclass $\mathcal G_{\geq 1}$. In \cite{BAYILLI}, the Nevanlinna counting function of $\varphi=c_0s+\psi\in\mathcal G_{\geq 1}$ was defined on $\CC_0$
by
$$\mathcal N_\varphi(w)=\sum_{\varphi(s)=w}\Re e(s),\ w\in\CC_0.$$
In that paper, it was shown that provided $|\Im m (\psi)|$ is bounded, the condition $\mathcal N_\varphi(w)=o(\Re e(w))$ as $\Re e(w)$ tends to $0$ implies the compactness of $C_\varphi$. Conversely,
in \cite{Bail15}, Bailleul established that if $\psi$ is supported on a finite set of prime numbers and is finitely valent, the compactness of $C_\varphi$ implies the above condition
on $\mathcal N_\varphi$. Moreover these two properties are equivalent if $\psi$ is supported on a single prime number (see \cite{BPortho}). We recall that if $\mathcal Q$ is a subset of the set
of prime numbers, a Dirichlet series $\sum_n a_n n^{-s}$ is supported on $\mathcal Q$ provided $a_n=0$ as soon as there exists a prime number not in $\mathcal Q$ such that $p|n$.

Our aim here is to show that the result of \cite{BAYILLI} is still true without any assumption 
on $|\Im m (\psi)|$.

\begin{theorem}\label{thm:main}
 Let $\varphi\in\Gu$ and let us assume that $\mathcal N_\varphi(w)=o(\Re e(w))$ as $\Re e(w)$ tends to $0$.
 Then $C_\varphi$ is compact on $\mathcal H^2$.
\end{theorem}

\subsection{Background material}

We briefly review some basic facts on Dirichlet series. Let $\TT^\infty$ be the infinite polycircle endowed with its Haar measure $m$. It can be identified
to the group of characters of $(\QQ_+,+)$ via the prime number factorization: to any $z\in\TT^\infty$
we associate the character $\chi$ defined by
$$\chi(n)=z_1^{\alpha_1}\cdots z_d^{\alpha_d}\textrm{ for }n=\prod_{j=1}^d p_j^{\alpha_j}.$$
For $f(s)=\sum_n a_n n^{-s}$ a Dirichlet series and $\chi$ a character, we denote by $f_\chi$ the Dirichlet series $f_\chi(s)=\sum_n a_n \chi(n)n^{-s}$. 
If $f$ converges uniformly in $\overline{\CC_\theta}$ for some $\theta\in\RR$, then for any $\chi\in\TT^\infty$, there exists a sequence of real numbers $(\tau_n)$
such that $f_\chi$ is the uniform limit in $\overline{\CC_\theta}$ of the vertical translates $(f(\cdot+i\tau_n))$. Conversely all uniform limits in $\overline{\CC_\theta}$
of vertical translates are equal to some $f_\chi$, which justifies that the functions $f_\chi$ are called vertical limit functions.

If we now assume that $f$ belongs to $\mathcal H^2$, then for almost every $\chi\in\TT^\infty$, $f_\chi$ converges in $\CC_0$ and one can compute the norm of $f$
via the following Littlewood-Paley formula:
$$ \mu(\mathbb R)\|f\|^2=\mu(\RR)|f(+\infty)|^2+4\int_{\TT^\infty}\int_0^{+\infty}\int_{\RR} \sigma |f_\chi'(\sigma+it)|^2 d\mu(t)d\sigma dm(\chi),$$
where $\mu$ is any finite positive measure on $\RR$. 

Regarding composition operators, the notion of vertical limits is extended to symbols $\varphi\in\mathcal G_{\geq 1}$ by defining 
$\varphi_\chi=c_0 s+\psi_\chi$. The composition operators $C_{\varphi}$ and $C_{\varphi_\chi}$ are related by the formula
$$(f\circ\varphi)_\chi=f_{\chi^{c_0}}\circ\varphi_\chi.$$

\section{On counting functions for Dirichlet series}

Let $\varphi$ belong to $\Gu$. The classical Nevanlinna counting function associated to $\varphi$ is defined on $\CC_0$ by
$$\Nphi(w)=\sum_{\varphi(s)=w}\Re e(s),\ w\in \CC_0.$$
It was introduced in \cite{BAYILLI} where the two following important properties were proved:
\begin{itemize}
 \item[{\bf (NC1)}] $\Nphi(w)\leq \frac 1{c_0}\Re e(w)$ for all $w\in\CC_0$.
 \item[{\bf (NC2)}] If $\Nphi(w)=o(\Re e(w))$ as $\Re e(w)\to 0$, then for all $\veps>0$, there exists $\theta>0$ such that, for all $w\in\CC_0$
 with $\Re e(w)<\theta$, for all $\chi\in\TT^\infty$, $\mathcal N_{\varphi_\chi}(w)\leq \veps \Re e(w)$ (namely, 
 the $o$ bound is uniform with respect to $\chi\in\TT^\infty$).
\end{itemize}

In this paper, we will be also interested in a restricted version of the counting function, which is defined by
$$N_\varphi(w)=\sum_{\substack{\varphi(s)=w\\ |\Im m(s)|<1}}\Re e(s),\ w\in\CC_0.$$
A similar restricted counting function has been introduced in \cite{BPortho}, when $\varphi$ is supported on a single prime number.
The main interest of working with $N_\varphi$ instead of $\Nphi$ comes from the following enhancement of (NC1), which is inspired by
\cite[Lemma 6.3]{BPortho}.

\begin{proposition}\label{prop:uniformnphi}
 Let $\varphi=c_0s+\psi\in\Gu$. There exists $C>0$ such that, for all $\chi\in\TT^\infty$, for all $w\in\CC_0$ with $\Re e(w)<c_0$, 
 $$N_{\varphi_\chi}(w)\leq C\frac{\Re e(w)}{1+(\Im m(w))^2}.$$
\end{proposition}
\begin{proof}
 Let $\Theta$ be the conformal map from $\mathbb D$ onto the half-strip
 $$S=\{s=\sigma+it:\ \sigma>0,\ |t|<2\}$$
 normalized by $\Theta(0)=2$ and $\Theta'(0)>0$. By standard regularity results on conformal maps there exists $C_1>0$ 
 such that, for all $s$ with $0<\Re e(s)<1$ and $|\Im m(s)|<1$, 
 \begin{equation}\label{eq:uniformnphi1}
 \Re e(s)\leq C_1  \log\frac1{|\Theta^{-1}(s)|}.
 \end{equation}
 Fix $w\in\CC_0$ and $\chi\in\TT^\infty$ with $0<\Re e(w)<c_0$. Define $G$ on $\DD$ by
 \begin{equation}\label{eq:uniformnphi2}
 G(z)=G_{w,\chi}(z):=\frac{\varphi_\chi(\Theta(z))-w}{\varphi_{\chi}(\Theta(z))+\bar w}
 \end{equation}
 which is a self-map of $\DD$. The Littlewood inequality for the standard Nevanlinna counting function of $G$
 says that 
 \begin{equation}\label{eq:uniformnphi3}
  \sum_{z\in G^{-1}(\{0\})} \log\frac 1{|z|}\leq \log\frac 1{|G(0)|}.
 \end{equation}
 Now $G(z)=0$ if and only if $\varphi_\chi(\Theta(z))=w$ so that \eqref{eq:uniformnphi3} becomes 
 $$\sum_{\Theta(z)\in\varphi_\chi^{-1}(\{w\})}\log \frac1{|z|}\leq \log\left|\frac{\varphi_\chi(\Theta(0)+\bar w}{\varphi_\chi(\Theta(0))-w}\right|$$
 which itself can be rewritten
 $$\sum_{s\in\varphi_\chi^{-1}(\{w\})}\log\frac1{|\Theta^{-1}(s)|}\leq \log \left|\frac{\varphi_\chi(2)+\bar w}{\varphi_\chi(2)-w}\right|.$$
Observe now that, when $\varphi_\chi(s)=w$, then $0<\Re e(s)<1$ so that, using \eqref{eq:uniformnphi1},
\begin{align*}
 N_{\varphi_\chi}(w)&=\sum_{\substack{\varphi_\chi(s)=w\\ |\Im m(s)|<1,\ \Re e(s)<1}}\Re e(s)\\
 &\leq C_1\sum_{s\in\varphi_\chi^{-1}(\{w\})}\log\frac1{|\Theta^{-1}(s)|}\\
 &\leq C_1 \log\left |\frac{\varphi_\chi(2)+\bar w}{\varphi_\chi(2)-w}\right|.
\end{align*}
Now we apply \cite[Lemma 2.3]{BP21} to get 
\begin{align*}
 N_{\varphi_\chi}(w)&\leq 2C_1 \frac{\Re e (\varphi_\chi(2))\Re e(w)}{|w-\varphi_\chi(2)|^2}\\
 &\leq 2C_1 \frac{\Re e(\varphi_\chi(2))\Re e(w)}{(\Re e (\varphi_\chi(2))-\Re e(w))^2+(\Im m(\varphi_\chi(2))-\Im m(w))^2}.
\end{align*}
Our restriction on $\Re e(w)$ shows that $\Re e(\varphi_\chi(2))-\Re e(w)\geq c_0$. On the other hand, since $\psi$ 
is a Dirichlet series uniformly convergent in each $\CC_\veps$, $\veps>0$, $\psi(2+i\mathbb R)$ is bounded. 
$\psi_\chi$ being a vertical limit function of $\psi$, there exists $C_2>0$ such that $|\varphi_\chi(2)|\leq C_2$
for all $\chi\in\TT^\infty$. If $|\Im m(w)|\leq 2C_2$, then we write
$$N_{\varphi_\chi}(w)\leq \frac{2C_1C_2\Re e(w)}{c_0^2}\leq \frac{2C_1C_2}{c_0^2}(1+4C_2^2)\frac{\Re e(w)}{1+(\Im m(w))^2}$$
whereas, if $|\Im m(w)|>2C_2$, then $\Im m(w)-\Im m(\varphi_\chi(2))\geq \frac 12\Im m(w)$ which yields
$$N_{\varphi_\chi}(w)\leq 2C_1C_2 \frac{\Re e(w)}{c_0+\frac 14(\Im m(w))^2}.$$
Therefore, in all cases, Proposition \ref{prop:uniformnphi} is proved.
\end{proof}

Proposition \ref{prop:uniformnphi} will be used to give a uniform bound on $N_{\varphi_\chi}$, in the spirit of (NC2),
upon the assumption $\Nphi(w)=o(\Re e(w))$.
\begin{corollary}\label{cor:ncuniform}
 Let $\varphi\in\Gu$ and $\delta\in(0,1)$. Let us assume that $\Nphi(w)=o(\Re e (w))$ as $\Re e(w)$ goes to $0$.
 Then for all $\veps>0$ there exists $\theta>0$ such that, for all $\chi\in\TT^\infty$, for all $w\in\CC_0$ with $\Re e(w)<\theta$,
 \begin{equation}\label{eq:ncuniform}
   N_{\varphi_\chi}(w)\leq \veps \frac{\Re e(w)}{\big(1+(\Im m(w))^2\big)^{(1+\delta)/2}}.
 \end{equation}
\end{corollary}
\begin{proof}
 Assume that the result is false. Then we can find $\veps>0$, a sequence $(w_k)$ in $\CC_0$ with $\Re e(w_k)\to 0$ and a sequence $(\chi_k)$ in $\TT^\infty$
 such that, for all $k$, 
 $$\mathcal N_{\varphi_{\chi_k}}(w_k)\geq N_{\varphi_{\chi_k}}(w_k)>\veps \frac{\Re e(w_k)}{\big(1+(\Im m(w_k))^2\big)^{(1+\delta)/2}}.$$
 If $\Im m(w_k)$ is unbounded, this contradicts Proposition \ref{prop:uniformnphi} whereas if $\Im m(w_k)$ is bounded, this contradicts that $\mathcal N_\varphi(w_k)=o(\Re e(w_k))$
 and in particular (NC2).
 \end{proof}

\section{Compact composition operators}

\subsection{Compactness on Hardy spaces}

\begin{proof}[Proof of Theorem \ref{thm:main}]
 Let $(f_n)$ be a sequence in $\mathcal H^2$ that converges weakly to zero. We will let $f_{n,\chi}$ the vertical limit function of $f_n$ with respect
 to the character $\chi$. By the Littlewood-Paley formula applied with $d\mu(t)=\frac 12\mathbf 1_{[-1,1]}dt$, setting $s=\sigma+it$, 
 $$\|C_\varphi(f_n)\|^2=|f_n(+\infty)|^2+2\int_{\TT^\infty}\!\int_{\RR_+}\int_{-1}^1 |f_{n,\chi^{c_0}}'(\varphi_\chi(s))|^2 |\varphi_\chi'(s)|^2 \sigma dtd\sigma dm(\chi).$$
 Our assumption on $(f_n)$ implies that $(f_n(+\infty))$ tends to zero. In the inner-most integrals we do the non-univalent change of variables $w=u+iv=\varphi_\chi(\sigma+it)$.
 Observe that this change of variables involves the restricted Nevanlinna counting function $N_\varphi$ whereas in \cite{BAYILLI} we used
 $\mathcal N_\varphi$ (but we only obtained an inequality). Hence 
 $$\|C_\varphi(f_n)\|^2=|f_n(+\infty)|^2+2\int_{\TT^\infty}\!\int_{\RR_+}\int_{\mathbb R}|f_{n,\chi^{c_0}}'(w)|^2 N_{\varphi_\chi}(w)dvdudm(\chi).$$
 Now let $\veps>0$ and let $\theta>0$ be given by Corollary \ref{cor:ncuniform} for $\delta=1/2$. We split the integral over $\RR_+$
 into $\int_0^{\theta}+ \int_{\theta}^{+\infty}$. For the first integral, say $I_0$, we use \eqref{eq:ncuniform} to get
 \begin{align*}
  I_0&\leq  \veps\int_{\TT^\infty}\int_{\RR_+}\int_{\RR} |f_{n,\chi^{c_0}}'(w)|^2 u \frac{dv}{(1+v^2)^{3/4}}dudm(\chi)\\
  &\leq C_1 \veps \|f_n\|^2
 \end{align*}
where we have used the Littlewood-Paley equality with the finite measure $d\mu(v)=\frac{dv}{(1+v^2)^{3/4}}$. To hande the integral over $[\theta,+\infty[$,
say $I_2$, we now use Proposition \ref{prop:uniformnphi} to write
\begin{align*}
 I_1&\leq C\int_{\TT^\infty}\int_{\theta}^{+\infty}\int_{\RR} |f_{n,\chi^{c_0}}'(w)|^2 \frac{u}{1+v^2} dvdudm(\chi)\\
 &\leq C\int_{\theta}^{+\infty}\int_{\RR}\sum_n |a_{n,k}|^2(\log k)^2 k^{-2u}\frac{u}{1+v^2}dvdudm(\chi)\\
 &\leq C_2  \sum_{k\geq 1}|a_{n,k}|^2 (\log^2 k)\int_{\theta}^{+\infty} k^{-2u}du
\end{align*}
where we have written $f_n(s)=\sum_{k\geq 1}a_{n,k} k^{-s}$. Now the argument of \cite{BAYILLI} shows that 
$\limsup_{n\to+\infty}  \sum_{k\geq 1}|a_{n,k}|^2 (\log^2 k)\int_{\theta}^{+\infty} k^{-2u}du=0$,
which finishes the proof of Theorem \ref{thm:main}.
\end{proof}

\begin{remark}
 Using the Littlewood-Paley formula for $\mathcal H^p$ and the argument of \cite{BQS16}, a variant of the above proof
 shows that the condition $\mathcal N_\varphi(w)=o(\Re e(w))$ as $\Re e(w)$ tends to zero is also sufficient to prove
 the compactness of $C_\varphi$ on $\mathcal H^p$, $p\geq 1$.
\end{remark}

\begin{remark}
 Proposition \ref{prop:uniformnphi} provides also an alternative approach to the continuity of $C_\varphi$ when $\varphi\in\Gu$.
 Nevertheless, as it is written, we loose the fact that $C_\varphi$ is a contraction.
\end{remark}

\subsection{Compactness on Bergman spaces}

In this section, we turn to the Bergman spaces of Dirichlet series introduced by McCarthy in \cite{McCarthy04}. For $\alpha>-1$ define
$$\mathcal A_\alpha=\left\{f(s)=\sum_{n=1}^{+\infty}a_n n^{-s}:\ \|f\|_\alpha^2=\sum_{n=1}^{+\infty} \frac{|a_n|^2}{(1+\log n)^{1+\alpha}}<+\infty\right\}$$
(as in the case of the disc, the Hardy space $\mathcal H^2$ corresponds to the limiting case $\alpha=-1$). 
The norm of an element $f$ of $\mathcal A_\alpha$ can be evaluated thanks to the following Littlewood-Paley formula:
$$\|f\|_\alpha^2\asymp  |f(+\infty)|^2+\int_{\TT^\infty}\int_0^{+\infty}\int_{\RR} |f_\chi'(\sigma+it)|^2 \sigma^{2+\alpha}d\mu(t) d\sigma dm(\chi)$$
where $\mu$ is any finite positive measure on $\RR$. Moreover, if $f$ is a Dirichlet series which converges uniformly in each half-plane $\CC_\veps$, $\veps>0$, 
an easy adaptation of \cite[Lemma 2.2]{BP21} shows that 
\begin{equation}\label{eq:LPlimitbergman}
 \|f\|_\alpha^2\asymp |f(+\infty)|^2+\lim_{\sigma_0\to 0^+}\lim_{T\to+\infty}\frac 1T\int_{\sigma_0}^{+\infty}\int_{-T}^T |f'(\sigma+it)|^2\sigma^{2+\alpha}dtd\sigma.
\end{equation}

Compactness of composition operators on $\mathcal A_\alpha$ has been studied in \cite{BAYILLI} and in \cite{Bail15}. The interest of working in Bergman spaces
is that one often can replace a sufficient condition on some counting function by a sufficient condition on the symbol itself. This is what happens again here. 
\begin{theorem}\label{thm:bergman}
Let $\varphi\in\mathcal G$ and $\alpha>-1$. 
\begin{enumerate}[a)]
\item If $\varphi\in\mathcal G_{\geq 1}$ and $\frac{\Re e(\varphi(s))}{\Re e(s)}\to+\infty$
as $\Re e(s)\to 0$, then $C_\varphi$ is compact on $\mathcal A_\alpha$.
\item If $\varphi\in\mathcal G_0$ and $\frac{\Re e(\varphi(s))-1/2}{\Re e(s)}\to+\infty$
as $\Re e(s)\to 0$, then $C_\varphi$ is compact on $\mathcal A_\alpha$.
\end{enumerate}
\end{theorem}
\begin{proof}
Let us start with a). Let $N_{\alpha,\varphi}=\sum_{\substack{\varphi(s)=w \\ |\Im m(s)|<1}}\big(\Re e(s)\big)^{2+\alpha}$ be the appropriated Nevanlinna counting function for $\mathcal A_\alpha$. We first observe that, for all $\chi\in\TT^\infty$ and for all $w\in\CC_0$, 
\begin{align*}
N_{\alpha,\varphi_\chi}(w)&\leq \left(\sum_{\substack{\varphi_\chi(s)=w \\ |\Im m(s)|<1}}\Re e(s)\right)^{2+\alpha}\\
&\leq C \frac{(\Re e(w))^{2+\alpha}}{1+(\Im m(w))^{2+\alpha}}
\end{align*}
for some $C>0$. Moreover, let $\veps>0$ and $\theta>0$ be such that $\Re e(s)<\theta\implies \Re e(s)<\veps \Re e(\varphi(s))$. For all $\chi\in\TT^\infty$, one gets by vertical limit $\Re e(s)\leq \veps \Re e(\varphi_\chi(s))$. Let $w\in\CC_0$ with $\Re e(w)<\theta$ and let $\chi\in\TT^\infty$. 
If $\varphi_{\chi}^{-1}(\{w\})=\varnothing$ then $N_{\alpha,\varphi_\chi}(w)=0$. Otherwise, any $s\in\varphi_{\chi}^{-1}(\{w\})$ satisfy $\Re e(s)<\theta$ so that 
\begin{align*}
N_{\alpha,\varphi_\chi}(w)&=\sum_{\substack{\varphi_\chi(s)=1\\ |\Im m(s)|<1}}\big(\Re e(s)\big)^{2+\alpha}\\
&\leq \veps^{\alpha+1}\big(\Re e(w)\big)^{\alpha+1}\sum_{\substack{\varphi_\chi(s)=1\\ |\Im m(s)|<1}}\Re e(s)\\
&\leq C\veps^{\alpha+1}\frac{\big(\Re e(w)\big)^{2+\alpha}}{1+(\Im m(w))^2}
\end{align*}
where we have used Proposition \ref{prop:uniformnphi}. We then conclude exactly as in the proof of Theorem \ref{thm:main}. Details are left to the reader. 

Let us turn to b). We are inspired by \cite{BP21} but working in a Bergman space and using this very strong assumption we will avoid most of the technical
difficulties which appear here. First we observe that for any $f\in\mathcal A_\alpha$, $\sigma_u(f\circ \varphi)\leq 0$, which implies that we can estimate the norm
of $f\circ\varphi$ using \eqref{eq:LPlimitbergman}. We then consider, for $\sigma_0,T>0$ and $w\in\CC_0$, the counting function 
$$\mathcal M_{\alpha,\varphi}(\sigma_0,T;w)=\frac 1T\sum_{\substack{\varphi(s)=w\\ |\Im m(s)|<T,\ \Re e(s)>\sigma_0}}\big(\Re e(s)\big)^{2+\alpha}.$$
The nonunivalent change of variables $w=\varphi_\chi(s)$ leads to 
$$\|C_\varphi f\|_\alpha^2\asymp |f(\varphi(+\infty))|^2+\lim_{\sigma_0\to 0^+}\lim_{T\to+\infty}\int_{\CC_{1/2}}|f'(w)|^2 \mathcal M_{\alpha,\varphi}(\sigma_0,T;w)d\sigma dt.$$
Let $\veps>0$ and $0<\theta<\Re e(\varphi(+\infty))/2$ such that $\Re e(\varphi(s))<\frac12+\theta$ implies $\Re e(s)<\veps\left(\Re e(\varphi(s))-\frac 12\right)$. Then 
\begin{align*}
 \mathcal M_{\alpha,\varphi}(\sigma_0,T;w)&\leq \frac 1T \sum_{\substack{\varphi(s)=w\\| \Im m(s)|<T}} \big(\Re e(s)\big)^{2+\alpha}\\
 &\leq \veps^{1+\alpha}\left(\Re e(w)-\frac12\right)^{1+\alpha}\frac 1T\sum_{\substack{\varphi(s)=w\\ | \Im m(s)|<T}} \Re e(s)\\
 &\leq C \frac{\veps^{1+\alpha}  \left(\Re e(w)-\frac 12\right)^{2+\alpha}}{|w-\varphi(+\infty)|^2}
\end{align*}
by \cite[Lemma 2.3 and 2.4]{BP21}
where the constant $C$ is uniform in $T$ for all $T\geq\sigma_1$, for some $\sigma_1>0$. 
The compactness of $C_\varphi$ now follows from an argument similar to that of \cite[Theorem 1.4]{BP21}, using that the estimate on $\mathcal M_{\alpha,\varphi}(\sigma_0,T;w)$ 
for $\Re e(w)<\theta$ is uniform with respect to $\sigma_0$ and $T$.
\end{proof}

\begin{corollary}
 Let $\varphi\in\mathcal G$ and $\alpha>-1$. 
 \begin{enumerate}
  \item If $\varphi\in\mathcal G_{\geq 1}$ and $\varphi$ is supported on a finite set of prime numbers, then 
  $$C_\varphi\textrm{ is compact on }\mathcal A_\alpha \iff \frac{\Re e(\varphi(s))}{\Re e(s)}\xrightarrow{\Re e(s)\to 0}+\infty.$$
  \item If $\varphi\in\mathcal G_{0}$ and $\varphi$ is supported on a single prime number, then 
  $$C_\varphi\textrm{ is compact on }\mathcal A_\alpha \iff \frac{\Re e(\varphi(s))-1/2}{\Re e(s)}\xrightarrow{\Re e(s)\to 0}+\infty.$$
 \end{enumerate}
\end{corollary}

\begin{remark}
 For the case of positive characteristic, we do not know whether the condition $\Re e(\varphi(s))/\Re e(s)\to+\infty$ is always necessary for compactness.
\end{remark}

\subsection{Concluding remarks}

\begin{question}
 On $\mathcal H^2$, can we get a necessary condition using some counting function without any extra assumption on $\varphi\in\mathcal G_{\geq 1}$?
 Or at least, if $\varphi\in\mathcal G_{\geq 1}$ is supported on a finite set of prime numbers, do we have
 $$C_\varphi\textrm{ is compact on }\mathcal H^2\iff \mathcal N_\varphi(w)=o\big(\Re e(w)\big)\textrm{ as }\Re e(w)\to 0?$$
\end{question}

\begin{remark}
 The condition $ \frac{\Re e(\varphi(s))-1/2}{\Re e(s)}\xrightarrow{\Re e(s)\to 0}+\infty$ is not necessary for compactness on $\mathcal A_\alpha$
 when the symbol is not supported on a single prime number. For instance, $\varphi(s)=\frac 52-2^{-s}-3^{-s}$ generates a compact operator
 on $\mathcal A_\alpha$ (the proof of \cite[Theorem 2]{BB16}, done in $\mathcal H^2$, can be adapted to $\mathcal A_\alpha$) 
 which does not satisfy the above condition.
\end{remark}

\begin{remark}
 We can also use the restricted Nevanlinna counting function introduced here to simplify the results of \cite{BWY22}, deleting an unnecessary assumption of boundedness.
 For instance, we can replace Theorem 5.5 of \cite{BWY22} by: let $\varphi_0$ and $\varphi_1\in\mathcal G_{\geq 1}$ and
write them $\varphi_0=c_0s+\psi_0$, $\varphi_1=c_0s+\psi_1$. Assume moreover that there exists $C>0$ such that
\begin{itemize}
\item $|\varphi_0-\varphi_1|\leq C\min(\A \varphi_0,\A \varphi_1)$;
\item $|\varphi_0'-\varphi_1'|\leq C$.
\end{itemize}
Then $C_{\varphi_0}$ and $C_{\varphi_1}$ belong to the same component of $\mathcal C(\mathcal H)$, the set of composition operators on $\mathcal H$. If moreover we assume that
\[ |\varphi_0-\varphi_1|=o\big(\min(\A\varphi_0,\A \varphi_1)\big)\textrm{ and }\varphi_0'-\varphi_1'\to 0\textrm{ as }\min(\A\varphi_0,\A\varphi_1)\to 0 \]
then $C_{\varphi_0}-C_{\varphi_1}$ is compact. 
\end{remark}

\providecommand{\bysame}{\leavevmode\hbox to3em{\hrulefill}\thinspace}
\providecommand{\MR}{\relax\ifhmode\unskip\space\fi MR }
\providecommand{\MRhref}[2]{%
  \href{http://www.ams.org/mathscinet-getitem?mr=#1}{#2}
}
\providecommand{\href}[2]{#2}

\end{document}